\documentclass[a4paper,10pt]{article}

\usepackage{amsfonts}
\usepackage{amssymb}
\usepackage{amsthm}
\usepackage[english]{babel}
\usepackage{hhline}
\usepackage[utf8]{inputenc}
\usepackage{amsmath}
\usepackage[all,2cell,ps]{xy}
\usepackage{qsymbols}
\usepackage{color}
\usepackage{epsfig}
\usepackage{color}
\usepackage{graphics}
\usepackage{graphicx}%
\usepackage{enumerate}
\usepackage{tikz}
\usepackage{amssymb}
\usepackage[bottom]{footmisc}
\usepackage{authblk}

\usepackage{color}

\theoremstyle{plain}
\newtheorem{thm}{Theorem}
\newtheorem{cor}[thm]{Corollary}
\newtheorem{lem}[thm]{Lemma}
\newtheorem{defn}{Definition}
\newtheorem{rem}{Remark}
\newtheorem{prop}{Proposition}
\newtheorem{ex}{Example}
\newcommand{\ra}{\rightarrow}

\newcommand{\C}{\mathsf{C}}

\newcommand{\we}{\wedge}

\newcommand{\ConA}{\mathrm{Con}(A)}
\newcommand{\SU} {\mathrm{Sub_C}(A)}
\newcommand{\SUs} {\mathrm{Sub_{SC}}(A)}

\newcommand{\CRL}{\mathsf{CRL}}
\newcommand{\SRL}{\mathsf{SRL}}

\newcommand{\gSRL}{\mathsf{SR}}
\newcommand{\gSRLc}{\mathsf{SR}^{\mathrm{c}}}

\newcommand{\igSRL}{\mathsf{iSR}}

\newcommand{\Co}{\mathrm{C}_1}
\newcommand{\Ct}{\mathrm{C}_2}

\newcommand{\Eo}{\mathrm{E_1}}
\newcommand{\Et}{\mathrm{E_2}}

\title{Subresiduated lattice ordered commutative monoids}

\date{}

\author{Cornejo J.M., San Mart\'{\i}n H.J. and S\'{\i}gal V.}

\begin{document}

\maketitle

\begin{abstract}
A subresiduated lattice ordered commutative monoid
(or srl-monoid for short) is a pair
$(\textbf{A},Q)$ where $\textbf{A}=(A,\we,\vee,\cdot,e)$ is an
algebra of type $(2,2,2,0)$ such that $(A,\we,\vee)$ is a lattice,
$(A,\cdot,e)$ is a commutative monoid, $(a\vee b)\cdot c = (a\cdot
c) \vee (b\cdot c)$ for every $a,b,c\in A$ and $Q$ is a subalgebra
of \textbf{A} such that for each $a,b\in A$ there exists $c\in Q$
with the property that for all $q\in Q$, $a\cdot q \leq b$ if and
only if $q\leq c$. This $c$ is denoted by $a\ra_Q b$, or simply by
$a\ra b$.

The srl-monoids $(\textbf{A},Q)$ can
be regarded as algebras $(A,\we,\vee,\cdot,\ra,e)$ of type $(2,2,2,2,0)$. These algebras
are a generalization of subresiduated lattices and commutative
residuated lattices respectively.

In this paper we prove that the class of srl-monoids forms
a variety. We show that the lattice of congruences of any
srl-monoid is isomorphic to the lattice of its strongly convex subalgebras
and we also give a description of the strongly convex subalgebra generated by
a subset of the negative cone of any srl-monoid. We apply both results in order
to study the lattice of congruences of any srl-monoid by giving
as application alternative equational basis
for the variety of srl-monoids generated by its totally ordered members.
\end{abstract}

\section{Introduction}

A \emph{lattice ordered monoid} (or l-monoid for short) is an algebra $(A,\we,\vee,\cdot,e)$
of type $(2,2,2,0)$ such that $(A,\we,\vee)$ is a lattice, $(A,\cdot,e)$ is a
monoid and the equalities $(a\vee b)\cdot c = (a\cdot c) \vee (b\cdot c)$
and $c \cdot (a\vee b) = (c\cdot a) \vee (c\cdot b)$ are satisfied for every $a,b,c\in A$
\footnote{In \cite{J} the terminology \emph{lattice ordered monoid} is used in a different
sense to that employed in the present paper. More precisely, in \cite{J}
a \emph{lattice ordered monoid} is defined as an algebra $(A,\we,\vee,\cdot,e)$ of type $(2,2,2,0)$
such that $(A,\we,\vee)$ is a lattice, $(A,\cdot,e)$ is a monoid and the equalities
$(a\vee b)\cdot c = (a\cdot c) \vee (b\cdot c)$, $c \cdot (a\vee b) = (c\cdot a) \vee (c\cdot b)$,
$(a\we b)\cdot c = (a\cdot c) \we (b\cdot c)$ and $c \cdot (a\we b) = (c\cdot a) \we (c\cdot b)$
are satisfied for every $a,b,c\in A$.}.
Note that in particular $\cdot$ is monotone, i.e., for every
$a,b,c\in A$, if $a\leq b$ then $a\cdot c \leq b\cdot c$ and $c \cdot a \leq c \cdot b$.
A \emph{lattice ordered commutative monoid} (or commutative l-monoid for short) is an l-monoid
$(A,\we,\vee,\cdot,e)$ where the binary operation $\cdot$ is commutative
\footnote{The definition of \emph{lattice ordered commutative monoid} used in the present paper
is exactly the definition of \emph{lattice-ordered monoid} given in \cite{W}.}.
\vspace{1pt}

Many varieties of interest for algebraic logic
are commutative l-monoids or have reduct in them.
Some relevant examples of algebras with reduct in
commutative l-monoids are commutative residuated lattices
and algebras with reduct in distributive lattices with a greatest element.
\vspace{1pt}

\begin{defn}\label{cd}
A subresiduated lattice ordered commutative monoid (or srl-monoid
for short) is a pair $(\textbf{A},Q)$ where $\textbf{A}=(A,\we,\vee,\cdot,e)$
is a commutative l-monoid and $Q$ is a subalgebra
of \textbf{A} such that for each $a,b\in A$ there exists $c\in Q$
with the property that for all $q\in Q$, $a\cdot q \leq b$ if and
only if $q\leq c$. This $c$ is denoted by $a\ra_Q b$, or simply by
$a\ra b$.
\end{defn}

Let $(\textbf{A},Q)$ be an srl-monoid. If there is not ambiguity
we write $(A,Q)$ in place of $(\textbf{A},Q)$.
The srl-monoids can be regarded as algebras
$(A,\we,\vee,\cdot,\ra,e)$ of type $(2,2,2,2,0)$.
Moreover, if $(A,Q)$ is an srl-monoid then $Q = \{a\in A:e\ra a = a\}$.
In this sense we say that the binary operation $\ra$ determines the set $Q$.

Note that if $(A,\we,\vee,\cdot,\ra,e)$ is an srl-monoid then for every $a,b\in A$ we
have that $a\ra b = \;\text{max}\;\{q\in Q: a\cdot q \leq b\}$. Conversely, if
$(A,\we,\vee,\cdot,e)$ is a commutative l-monoid
and $Q$ is a subalgebra of $(A,\we,\vee,\cdot,e)$ such that for every $a,b\in A$
there exists the maximum of the set $\{q\in Q: a\cdot q \leq b\}$ (which will be denoted by $a\ra b$)
then $(A,\we,\vee,\cdot,\ra,e)$ is an srl-monoid.

\begin{defn}
An srl-monoid $(A,\we,\vee,\cdot,\ra,e)$ is called integral
if it has greatest element, which is denoted by $1$, and $e = 1$.
\end{defn}

In \cite{EG} a logic for left continuous t-norms is introduced as a
general framework for fuzzy logic. Their corresponding algebraic structures are MTL-algebras,
which can be also presented as the variety of prelinear commutative integral bounded residuated
lattices.
This variety is generated as such by the class of its totally ordered members.
In this sense, the variety of prelinear integral (and bounded) srl-monoids, which is an
appropriate expansion (with bottom) of a subvariety of that of prelinear srl-monoids
studied in Section \ref{s4}, can be seen as a generalization of that of MTL-algebras.
Hence, its assertional logic generalizes monoidal t-norm based logic.
\vspace{1pt}

The definition of srl-monoid is motivated by the definitions of
subresiduated lattice (or sr-lattice for short) \cite{EH}
and commutative residuated lattice \cite{HRT} respectively.
We recall both definitions in what follows.
\vspace{1pt}

A \emph{subresiduated lattice} (or sr-lattice for short) is a pair $(A,Q)$, where $A$ is a bounded
distributive lattice, $Q$ is a bounded sublattice of $A$ and for
every $a,b\in A$ there is $c\in Q$ such that for all $q\in Q$,
$q\we a\leq b$ if and only if $q\leq c$. This $c$ is denoted by
$a\ra_Q b$, or simply by $a\ra b$.
Let $(A,Q)$ be a sr-lattice.
The pair $(A,Q)$ can be regarded as an algebra $(A,\we,\vee,\ra,0,1)$
of type $(2,2,2,0,0)$. Moreover, $Q=\{a\in A:1\ra a=a\}$.
The class of subresiduated lattices forms a variety \cite{EH} which properly
contains to the variety of Heyting algebras \cite{CNSM,SM0}.
If $(A,\we,\vee,\ra,0,1)$ is a subresiduated lattice and $a,b,c\in A$, then
the following condition is satisfied:
\[
\text{if}\; a\leq b\ra c\;\text{then}\;a\we b \leq c.
\]
However, the converse of the above mentioned property is not true in general.
\vspace{1pt}

An algebra $(A,\we,\vee,\cdot,\ra,e)$ of type $(2,2,2,2,0)$ is said to be
a \emph{commutative residuated lattice} if $(A,\cdot,e)$ is a commutative monoid,
$(A,\we,\vee)$ is a lattice and for every $a,b,c\in A$,
\[
a\cdot b\leq c\;\text{if and only if}\: a\leq b\ra c.
\]
The class of commutative residuated lattices forms a variety \cite{HRT}.
A commutative residuated lattice is called \emph{integral} if $a\leq e$
for every $a$.

\begin{rem}
If $(A,Q)$ is an srl-monoid then the algebra $(Q,\we,\vee, \cdot,\ra,e)$
is a commutative residuated lattice.
\end{rem}

We write $\gSRL$ for the class of srl-monoids, $\igSRL$ for the
class of integral srl-monoids, $\SRL$ for the variety of sr-lattices
and $\CRL$ for the variety of commutative residuated lattices respectively.
\vspace{3pt}

The following elemental properties establish the connection between sr-lattices, commutative
residuated lattices and srl-monoids:
\begin{itemize}
\item If $(A,\we,\vee,\ra,0,1)\in \SRL$ then $(A,\we,\vee,\we,\ra,1) \in \gSRL$
(in this sense we say that every sr-lattice is an srl-monoid). Moreover,
$(A,\we,\vee, \ra,0,1) \in \SRL$  if and only if $(A,\we,\vee,\we,\ra,1)\in \igSRL$
and this algebra has a least element.
\item If $(A,\we,\vee,\cdot,\ra,e)\in \CRL$ then $(A,\we,\vee,\cdot,\ra,e)\in \gSRL$
(it can be proved by defining $Q = A$). Moreover,
$(A,\we,\vee,\cdot,\ra,e)\in \CRL$ if and only if
$(A,\we,\vee,\cdot,\ra,e)\in \gSRL$ and $e\ra a = a$ for every $a\in A$.
\end{itemize}

In every srl-monoid $A$ the following condition is satisfied for every $a,b,c\in A$:
\[
\text{if}\; a\leq b\ra c\;\text{then}\; a\cdot b\leq c.
\]
The converse of the previous property is verified if and only if
$A\in \CRL$.
\vspace{1pt}

In what follows we will make some remarks about the logical counterpart of
the varieties $\SRL$ and $\CRL$ respectively.
\vspace{1pt}

S4-algebras are defined as Boolean algebras with a modal operator
$\square$ in the language that satisfies the identities $\square(1)=1$,
$\square(a\we b) = \square(a)\we \square(b)$, $\square(a) \leq a$ and
$\square(a)\leq \square(\square(a))$. It is known that the variety $\mathsf{S4}$,
whose members are the S4-algebras, is the algebraic semantics of the modal logic
$\mathbf{S4}$. This means that $\phi$ is a theorem of $\mathbf{S4}$ if and only
if the variety $\mathsf{S4}$ satisfies $\phi\approx 1$.
The variety $\SRL$ corresponds to the variety of algebras
defined for all the equations $\phi \approx 1$ satisfied
in the variety $\mathsf{S4}$ where only appears the connectives conjunction
$\we$, disjunction $\vee$, bottom $\perp$, top $\top$ and a new connective
of implication (called strict implication) defined by
$\varphi \Rightarrow \psi:= \square(\varphi\rightarrow\psi)$, where
$\ra$ denotes the classical implication \cite{CJ,EH}.
\vspace{1pt}

We write $\mathbf{HFL_e}^+$ for the system obtained from $\mathbf{HFL_e}$
removing the constant $0$ from its language, where $\mathbf{HFL_e}$ is defined as
the Hilbert system mentioned in \cite[Figure 2.9]{GJKO}. The variety $\CRL$
is the equivalent algebraic semantics for the consequence relation
$\vdash_{\mathbf{HFL_e}^+}$, where we recall that the translations
$\rho$ and $\tau$ from equations to formulas and from formulas to
equations respectively are defined by
$s \approx t\; \mapsto^{\rho}\; (s\ra t)\we (t\ra s)$ and
$\varphi\; \mapsto^{\tau}\; e \approx e\we \varphi$,
where $e$ denotes the constant in the signature of $\CRL$.
\vspace{1pt}

In this paper we focus our attention in the study of the classes $\gSRL$ and $\igSRL$,
which are varieties (see Theorem \ref{t1}). However, it would be also interesting to study the
features of the logic semantically defined by these varieties.
For the case of the variety $\igSRL$, there is a standard way to obtain a logic
which has this variety as its algebraic semantics.
We define $\mathsf{L}$ as the implicative logic, where $\ra$ denotes the implication, in the language $\{\we, \vee, \cdot,\ra\}$
with the additional axioms given by the translation of the set of equations which characterize the variety $\igSRL$
(obtained from Theorem \ref{t1} and the additional equation $x\we 1 \approx 1$), where the translation from equations to axioms is defined by
$s \approx t\; \mapsto\; (s\ra t)\we (t\ra s)$. Since $\mathsf{L}$ is an implicative logic then it
follows from \cite[Theorem 2.9]{Font} that $\mathsf{L}$ is complete with respect to the class of algebras
$\mathrm{Alg}^{*}(\mathsf{L})$ defined in \cite[Definition 2.5]{Font}, where the element $1$ is included in
the language of the algebras. Moreover, it follows from \cite[Proposition 2.7]{Font} that
$\mathrm{Alg}^{*}(\mathsf{L})$ coincides with the class of algebras in
the signature of integral srl-monoids given by the equations and quasi-equations that result
by applying the transformation $\varphi\; \mapsto\; \varphi \approx 1$ to the axioms and rules of $\mathsf{L}$.
It is immediate that $\mathrm{Alg}^{*}(\mathsf{L}) = \igSRL$, so $\mathsf{L}$ is complete with respect to $\igSRL$.
Furthermore, if we consider the logic $\mathsf{L}$ with the
additional axiom
\[
((\top \ra \alpha) \ra \alpha) \we (\alpha \ra (\top \ra \alpha))
\]
(where $\top$ is defined as $x\ra x$ by a fixed variable $x$), then we get that this logic is
complete with respect to the variety of integral commutative residuated lattices.
A more detailed study of the logical counterpart of the variety $\igSRL$, as well as a study of the
of the logical counterpart of the variety $\gSRL$, will be considered in a future paper.
In particular, we will give another presentation of a logic which is complete with respect to $\igSRL$.
\vspace{1pt}

There are other connections of integral srl-monoids with existing literature.
In \cite{C} Celani introduced distributive lattices with fusion
and implication, which are defined as algebras $(A,\we,\vee, \cdot,\ra,0,1)$ of type
$(2,2,2,2,0,0)$ such that $(A,\we,\vee,0,1)$ is a bounded distributive lattice and for every
$a,b,c\in A$ the following conditions are satisfied:
\begin{enumerate}
\item[(F1)] $a\cdot (b\vee c) = (a\cdot b) \vee (a\cdot c)$,
\item[(F2)] $(b\vee c)\cdot a = (b \cdot a) \vee (c\cdot a)$,
\item[(F3)] $0\cdot a = a\cdot 0 = 0$,
\item[(I1)] $(a\ra b)\we(a\ra c)=a\ra(b\we c)$,
\item[(I2)] $(a\ra c)\we(b\ra c)=(a\vee b)\ra c$,
\item[(I3)] $0\ra a = a \ra 1 = 1$.
\end{enumerate}
The variety of algebras with fusion and implication appears in a natural way in all algebraic
structures associated with many-valued logic and relevance logic.
A direct computation based in properties
of srl-monoids (see Section \ref{s1}) shows that if
$(A,\we,\vee,0,1)$ is a bounded distributive lattice and $(A,\we,\vee,\cdot,\ra,1) \in \igSRL$ then
$(A,\we,\vee,\cdot,\ra,0,1)$ is a distributive lattice with fusion and implication.
\vspace{1pt}

The aim of the present paper is to generalize some results on
subresiduated lattices and commutative residuated lattices respectively
in the framework of srl-monoids. Inspired by some results from \cite{CJ,EH}
we prove that the class of srl-monoids forms a variety and we present two
equational basis for it. Besides,
motivated by the results given in \cite{HRT}, we also study the lattice of congruences
of any srl-monoid and as an application we present alternative equational basis
for the variety of srl-monoids generated by its totally ordered members.
\vspace{1pt}

The paper is organized as follows. In Section \ref{s1} we give some basic
properties and examples of srl-monoids. We also prove that the class
of srl-monoids is a variety. In Section \ref{s2} we introduce the definition
of strongly convex subalgebra and we prove that
for $A \in \gSRL$ there exists an order isomorphism between
the lattice of congruences of $A$ and the lattice of strongly convex
subalgebras of $A$. This result is a natural generalization of
\cite[Theorem 2.3]{HRT}, which establishes an order isomorphism between the
lattice of congruences and the lattice of convex subalgebras of each
commutative residuated lattice. In Section \ref{s3} we give a description of
the strongly convex subalgebra generated by a subset of the negative cone
of any srl-monoid and we apply this result in order to give some properties
of the lattice of strongly convex subalgebras of each srl-monoid.
We finish this section by giving a characterization of the principal
congruences of any srl-monoid. Finally, in Section \ref{s4} we study the
variety generated by the class of srl-monoids whose order is total.
The main goal of this section is to give alternative equational basis for the
mentioned variety.

\section{Basic results} \label{s1}

In this section we give some preliminary properties of srl-monoids. Our
main goal is to show that the class $\gSRL$ is a variety.
\vspace{1pt}

We start by given some examples of srl-monoids.

\begin{ex} \label{ex1}
Let $A$ be a commutative l-monoid where the underlying lattice of $A$ is bounded
(the first element is denoted by $0$ and the greatest element is denoted by $1$).
Also assume that $e = 1$. Consider $Q$ as a subset of the universe of $A$ which
is a finite chain closed under $\cdot, 0$ and $1$.
Then for every $a$ and $b$ elements in the universe of $A$ we get $0\in \{q\in Q: a\cdot q \leq b\}$,
so this is a non empty set.
Thus, there exists the maximum of the set $\{q\in Q: a\cdot q \leq b\}$.
Therefore, $(A,Q)$ is an integral srl-monoid.
\end{ex}

\begin{ex} \label{ex1}
Consider the chain of three elements $A = \{0,e,1\}$ with $0<e<1$. We define
the binary operation $\cdot$ by

\[%
\begin{tabular}
[c]{c|ccc}%
$\cdot$ & $0$ &         $e$&             $1$ \\\hline
$0$     & $0$        &  $0$            & $0$ \\
$e$     & $0$        &  $e$            & $1$ \\
$1$     & $0$        &  $1$            & $1$ %
\end{tabular}
\hskip15pt
\]

Let $Q = \{0,e\}$. A direct computation shows that $(A,Q)$ is an srl-monoid.
Moreover, the operation $\ra$ is given by

\[%
\begin{tabular}
[c]{c|ccc}%
$\ra$ & $0$ &         $e$&             $1$ \\\hline
$0$   & $e$        &  $e$            & $e$ \\
$e$   & $0$        &  $e$            & $e$ \\
$1$   & $0$        &  $0$            & $e$ %
\end{tabular}
\]

Note that since $e\neq 1$, $(A,\we,\vee,\cdot,\ra,e) \notin \igSRL$.
Moreover, this algebra is not a sr-lattice because $e\cdot 1 \neq e\we 1$ and this is not
a commutative residuated lattice because $e\ra 1 \neq 1$.
\end{ex}

\begin{ex} \label{ex2}
Consider the chain of three elements $A = \{0,a,1\}$ with $0<a<1$. We define
binary operation $\cdot$ by

\[%
\begin{tabular}
[c]{c|ccc}%
$\cdot$ & $0$ &         $a$&             $1$ \\\hline
$0$     & $0$        &  $0$            & $0$ \\
$a$     & $0$        &  $0$            & $a$ \\
$1$     & $0$        &  $a$            & $1$ %
\end{tabular}
\hskip15pt
\]

Let $Q = \{0,1\}$. We have that $(A,Q)$ is an integral srl-monoid.
Moreover, the operation $\ra$ is given by

\[%
\begin{tabular}
[c]{c|ccc}%
$\ra$ & $0$ &         $a$&             $1$ \\\hline
$0$   & $1$        &  $1$            & $1$ \\
$a$   & $0$        &  $1$            & $1$ \\
$1$   & $0$        &  $0$            & $1$ %
\end{tabular}
\]

Note that the algebra $(A,\we,\vee,\cdot,\ra,1)$ is not a sr-lattice because $a\cdot a \neq a\we a$ and this is not
an integral commutative residuated lattice because $1\ra a \neq a$.
\end{ex}

\begin{ex} \label{ex3}
Let $A$ be the real interval $[0,1]$.
This set can be seen as an MV-algebra \cite{COM} and in consequence also as a commutative
residuated lattice. In particular,
$(a\vee b) \cdot c = (a\cdot c) \vee (b\cdot c)$, where we write $\cdot$ for the usual product in the MV-algebra
$[0,1]$.
Note that for each natural number $n$ with $n\geq 2$ we have that
$\L _n = \{0, \frac{1}{n-1},\frac{2}{n-1},\dots,\frac{n-2}{n-1},1\}$
is a subalgebra of the commutative residuated lattice $A$.
In what follows we also write $A$ for the $\{\we,\vee,\cdot,1\}$-reduct of the
mentioned commutative residuated lattice. We have that $(A,\L_n)$ is an integral srl-monoid.
\end{ex}

Now we give an equational characterization for the class of srl-monoids.

\begin{thm} \label{t1}
Let $(A,\we,\vee,\cdot,\ra,e)$ be an algebra of type
$(2,2,2,2,0)$.

Then $(A,\we,\vee,\cdot,\ra,e)$ is an srl-monoid if
and only if $(A,\we,\vee,\cdot,e)$ is a commutative l-monoid
and the following identities are satisfied:
\begin{enumerate}[\normalfont 1)]
\item $e\leq (x\wedge y)\ra y$,
\item $x\ra y \leq (z\wedge e)\ra (x\ra y)$,
\item $x\cdot (x\ra y) \leq y$,
\item $z\ra (x\wedge y) = (z\ra x)\wedge (z\ra y)$,
\item $e\ra ((e\ra x)\cdot(e\ra y)) = (e\ra x)\cdot (e\ra y)$,
\item $e\ra y \leq x \ra (x\cdot (e\ra y))$.
\end{enumerate}
\end{thm}

\begin{proof}
Suppose that $(A,\we,\vee,\cdot,\ra,e)$ is an srl-monoid.
It is immediate that the conditions 5) and 6) are verified.

Since $e\in Q$ and $e\cdot (a\wedge b) = a\wedge b \leq b$
then $e\leq (a\wedge b)\ra b$, which is 1). Taking into
account the definition of $a\ra b$ we
obtain $a\cdot (a\ra b) \leq b$, which is 3). Besides, since
$(c\wedge e)\cdot (a\ra b) \leq e\cdot (a\ra b) = a\ra b$ then $a\ra b \leq
(c\wedge e)\ra (a\ra b)$, which is 2). Now we will see that
\[
c \ra (a\we b) = (c\ra a) \we (c\ra b).
\]
By 3), $c\cdot (c\ra (a\we b)) \leq a\we b \leq
a$, so $c\ra (a\we b)\leq c\ra a$. In a similar way we have that
$c\ra (a\we b)\leq c\ra b$. Thus, $c\ra (a\we b)$ is a lower
bound of $\{c\ra a, c\ra b\}$ in $Q$. Let $d$ be a lower bound of
$\{c\ra a,c\ra b\}$ in $Q$. Then $d\in Q$, $d\leq c\ra a$ and
$d\leq c\ra b$, so $d\cdot c \leq a$ and $d\cdot c\leq b$. Hence,
$d\cdot c \leq a\we b$, which implies the inequality $d\leq c\ra (a\we b)$.
Hence, we have proved 4).

Conversely, suppose that $(A,\we,\vee,\cdot,e)$ is a
commutative l-monoid and that the identities 1),\dots, 6) are satisfied in $A$.
Let $Q = \{a\in A:e\ra a = a\}$ and $a,b,c \in A$. It follows
from 3) that
\begin{enumerate}
\item[I)] $e\ra a \leq a$.
\end{enumerate}
It follows from 4) that
\begin{enumerate}
\item[II)] if $a\leq b$ then $c\ra a\leq c\ra b$.
\end{enumerate}
We also have that
\begin{enumerate}
\item[III)] if $a\leq b$ then $a\cdot c\leq b\cdot c$.
\end{enumerate}

Suppose that $q,r\in Q$. It follows from I) and II) that
$q  =  e\ra q \leq  e\ra (q\vee r) \leq q\vee r$.
By the same reason, $r \leq e\ra (q\vee r) \leq q\vee r$.
Thus, $q\vee r = e\ra (q\vee r)$, i.e., $q\vee r \in Q$. By 4)
we also have that $q\we r \in Q$. By 1) and I) we have that $e\ra
e = e$, so $e\in Q$. By 5) we have that $e\ra (q\cdot r) = q\cdot
r$, so $q\cdot r \in Q$. Hence, $Q$ is a subalgebra of
$(A,\we,\vee,\cdot,e)$.

Let $a,b \in A$. By I) and 2) we have that
$e\ra (a\ra b)  \leq  a\ra b \leq  e\ra (a\ra b)$,
so $e\ra (a\ra b) = a\ra b$. Hence, $a\ra b \in Q$.

Finally, let $a,b\in A$ and $q\in Q$.
Assume that $a\cdot q\leq b$. It follows from 6) and II)
that
$q  =  e \ra q  \leq   a\ra (a\cdot q) \leq   a\ra b$,
so $q\leq a\ra b$. Now assume that $q\leq a\ra b$. Hence,
by III) and 3),
$a\cdot q   \leq   a\cdot (a\ra b) \leq   b$.
Thus, $a\cdot q\leq b$.
Therefore, $(A,\we,\vee,\cdot,\ra,e)$ is an srl-monoid.
\end{proof}

\begin{cor}\label{corvariety}
The class $\gSRL$ is a variety.
\end{cor}

It is interesting to note that if $(A,\we,\vee,\cdot,\ra,e)$ is an
algebra of type $(2,2,2,2,0)$ such that $(A,\we,\vee)$ is a lattice,
$(A,\cdot,e)$ is a monoid and $\cdot$ is a monotone map
then the following two conditions are equivalent:
\begin{enumerate}[\normalfont 1)]
\item $a\cdot (a\ra b) \leq b$ for every $a,b\in A$.
\item For every $a,b,c\in A$, if $a\leq b\ra c$ then $b\cdot a \leq c$.
\end{enumerate}
This property is a particular case of \cite[Proposition 1.1]{CSM}.
Therefore, it follows from Theorem \ref{t1} that condition 2) previously
mentioned is satisfied in any srl-monoid.

\begin{lem}\label{l1}
Let $(A,\we,\vee,\cdot,\ra,e)$ be an srl-monoid.

The following conditions are satisfied for every $a,b,c\in A$:
\begin{enumerate}[\normalfont 1)]
\item $(a\vee b) \ra c = (a\ra c) \wedge (b\ra c)$,
\item $(a\ra b)\cdot (b\ra c) \leq a\ra c$,
\item $e\leq a\ra a$, \item $a\leq b$ if and only if $e\leq a\ra b$,
\item $e\ra a \leq a$,
\item $e\ra (a\ra b) = a\ra b$,
\item $e\ra a \leq b \ra (a\cdot b)$.
\end{enumerate}
\end{lem}

\begin{proof}
1) Since
$a\cdot ((a\vee b)\ra c)  \leq  (a\vee b) \cdot ((a\vee b) \ra c) \leq   c$
then $(a\vee b)\ra c \leq a\ra c$. Similarly, $(a\vee b)\ra c \leq
b\ra c$. Thus, $(a\vee b)\ra c$ is a lower bound of $\{a\ra c,
b\ra c\}$ in $Q$. Let $d$ be a lower bound of $\{a\ra c, b\ra c\}$
in $Q$, so $d\in Q$, $d\leq a\ra c$ and $d\leq b\ra c$, so
$d\cdot a \leq c$ and $d\cdot b \leq  c$. Hence,
$d\cdot (a\vee b)   =  (d\cdot a) \vee (d\cdot b) \leq   c$.
Then $d\leq (a\vee b)\ra c$. Thus,
$(a\vee b) \ra c = (a\ra c) \wedge (b\ra c)$.

2) Since $a\cdot(a\ra b)\leq b$ and $b\cdot (b\ra c)\leq c$ then
$a\cdot (a\ra b) \cdot (b\ra c)  \leq  b \cdot (b\ra c) \leq   c$.
But $(a\ra b)\cdot (b\ra c) \in Q$, so
$(a\ra b)\cdot (b\ra c) \leq a\ra c$.

3) The inequality $e\leq a\ra a$ follows from that $e\in Q$ and $e\cdot a = a\leq a$.

4) Suppose $a\leq b$. Then $a\ra b \geq b\ra b$. But $b\ra b \geq e$, so $e\leq a\ra b$.
Conversely, assume that $e\leq a\ra b$. Hence,
$a =  a \cdot e \leq   a\cdot (a\ra b) \leq   b$.
Then $a\leq b$.

5) The inequality $e\ra a \leq a$ follows that $e\ra a = e\cdot (e\ra a) \leq a$.

6) It is immediate that the equality $e\ra (a\ra b) = a\ra b$ is satisfied.

7) Finally we will see that $e\ra a \leq b \ra (a\cdot b)$. First note that
$e\ra a \leq b \ra (b\cdot (e\ra a))$. Since $e\ra a \leq a$ then
$b\cdot (e\ra a) \leq a\cdot b$, so $b \ra (b\cdot (e\ra a)) \leq
b \ra (a\cdot b)$. Therefore, $e\ra a \leq b \ra (a\cdot b)$.
\end{proof}

Let $(A,\we,\vee,\cdot,\ra,e)$ be an algebra of type
$(2,2,2,2,0)$. For every $a\in A$ we define
\[
\square(a) = e\ra a.
\]

In the proof of the following corollary we will use Theorem \ref{t1}
and Lemma \ref{l1}.

\begin{cor}
Let $(A,\we,\vee,\cdot,\ra,e)$ be an algebra of type
$(2,2,2,2,0)$.

Then $(A,\we,\vee,\cdot,\ra,e)$ is an srl-monoid if
and only if $(A,\we,\vee,\cdot,e)$ is an l-commutative
monoid and the following identities are satisfied:
\begin{enumerate}[\normalfont 1)]
\item $z\ra (x\we y) = (z\ra x)\we (z\ra y)$,
\item $(x\vee y) \ra z = (x\ra z) \we (y\ra z)$,
\item $(x\ra y)\cdot (y\ra z) \leq x\ra z$,
\item $e\leq x\ra x$,
\item $x\cdot (x\ra y) \leq y$,
\item $x\ra y \leq (z\we e)\ra (x\ra y)$,
\item $\square(\square(x)\cdot \square(y)) = \square(x)\cdot \square(y)$,
\item $\square(y) \leq x \ra (x\cdot \square(y))$.
\end{enumerate}
\end{cor}

\begin{proof}
After Theorem \ref{t1} and Lemma \ref{l1} the implication from left to right is immediate.
Conversely, notice that conditions 1), 5), 6), 7) and 8) are exactly conditions
4), 3), 2), 5) and 6) of Theorem \ref{t1}, so we only need to prove the first equation
of Theorem \ref{t1}. Let $a,b \in A$.
Since $a\we b \leq b$ then $e \leq b\ra b \leq (a\we b) \ra b$,
so $e\leq (a\we b)\ra b$.
\end{proof}

In commutative residuated lattices it holds that
\begin{equation} \label{ewr}
a\cdot b \leq c\;\text{if and only if}\; a\leq b\ra c.
\end{equation}

Now we will prove a weak version of (\ref{ewr}) in the framework of srl-monoids.

\begin{prop} \label{c1}
Let $A\in \gSRL$ and $a,b,c\in A$. The following conditions are true:
\begin{enumerate}[\normalfont 1)]
\item If $a\leq b\ra c$ then $a\cdot b\leq c$.
\item If $a\cdot b \leq c$ then $\square(a) \leq b\ra c$.
\end{enumerate}
\end{prop}

\begin{proof}
Condition 1) is immediate.
In order to prove 2), suppose that $a\cdot b\leq c$. Then, by 7) of Lemma \ref{l1},
$\square(a)  \leq  b\ra (a\cdot b) \leq   b\ra c$,
so $\square(a)\leq b\ra c$.
\end{proof}

\section{Strongly convex subalgebras} \label{s2}

In this section we prove that for every $A\in \gSRL$ there exists an order isomorphism between
the lattice of congruences of $A$ and the lattice of strongly convex subalgebras of $A$.
\vspace{1pt}

We start with some preliminary results.
\vspace{1pt}

Let $A \in \gSRL$. We write $\ConA$ for the set of congruences of $A$.
If $\theta \in \ConA$ and $a\in A$, we write $a/\theta$ for the equivalence
class of $a$. Let $H$ be a \emph{convex subalgebra} of $A$, that is,
a subalgebra of $A$ such that for every $a,b,c\in A$,
if $a,b\in H$ and $a\leq c\leq b$ then $c\in H$.
We write $\SU$ for the set of convex subalgebras of $A$.
\vspace{3pt}

The following definition will play a fundamental role in this section.

\begin{defn}
Let $A \in \gSRL$. A subset $H$ of $A$ is said to be a strongly convex subalgebra of
$A$ if it is a convex subalgebra of $A$ such that for every $a\in A$ and $h\in H$ such
that $a\cdot h \leq e \leq h\ra a$ it holds that $a\in H$.
\end{defn}

We write $\SUs$ for the set of strongly convex subalgebras of $A$.
\vspace{3pt}

The following question naturally arises: is there any
$A\in \gSRL$ and $H\in \SU$ such that $H\notin \SUs$?
The answer is positive. Indeed, consider Example \ref{ex1}.
We have that $\{0,e\} \in \SU$ and this is not a strongly
convex subalgebra because $1\cdot 0 \leq e \leq 0\ra 1$ but $1\notin H$.

\vspace{3pt}
The aim of this section is to show that there exists an
order isomorphism between $\ConA$ and $\SUs$.

\begin{lem} \label{c2}
Let $A \in \gSRL$.
If $\theta \in \ConA$, then $e/\theta \in \SUs$.
\end{lem}

\begin{proof}
Let $A \in \gSRL$ and $\theta \in \ConA$.
The convexity of $e/\theta$ follows from the known fact that the equivalence blocks
of a lattice congruence are convex sublattices. We now prove that $e/\theta$ is a subalgebra.
To this end, let $a,b\in e/\theta$, i.e., $(a,e), (b,e)\in \theta$. It is immediate that
$(a\we b,e), (a\vee b,e)$ and $(a\cdot b,e)$ are elements of $\theta$. Since $e\ra e = e$ we
also have that $(a\ra b,e)\in \theta$. Thus, $e/\theta$ is a convex subalgebra of $A$.

Finally we will see that $e/\theta$ is a strongly convex subalgebra of $A$.
In order to show this, let $h\in e/\theta$ and $a\in A$ such
that $a\cdot h\leq e\leq h\ra a$, i.e., $a\cdot h \leq e$ and $h\leq a$. Since $(h,e)\in \theta$ then
$(a\cdot h \vee e, a\vee e) \in \theta$, i.e., $(e,a\vee e)\in \theta$. Again, since $(h,e)\in \theta$
then $(h\vee a,e\vee a)\in \theta$, i.e., $(a,e\vee a) \in \theta$. It follows from the symmetry and the
transitivity of $\theta$ that $(a,e)\in \theta$, which was our aim.
\end{proof}

Let $A\in \gSRL$ and $a,b\in A$. We define $s(a,b) = (a\ra b)\we (b\ra a)\we e$.

\begin{lem} \label{c3}
Let $A\in \gSRL$ and $\theta \in \ConA$. Then $(a,b)\in \theta$ if and only if $s(a,b)\in e/\theta$.
\end{lem}

\begin{proof}
Let $A\in \gSRL$ and $\theta \in \ConA$. Suppose that $(a,b)\in \theta$.
In consequence, $((a\ra a) \we e, (a\ra b) \we e)\in \theta$, i.e., $(e, (a\ra b)\we e)\in \theta$.
By the same reason, $(e,(b\ra a)\we e)\in \theta$. This implies that $(e,s(a,b))\in \theta$,
so $s(a,b)\in e/\theta$. Conversely, suppose that $s(a,b)\in e/\theta$.
Hence, $(a, a\cdot s(a,b))\in \theta$ and $(b,b\cdot s(a,b))\in \theta$.
Let $x = a\cdot s(a,b)$ and $y = b\cdot s(a,b)$. It is immediate that $x\leq b$ and $y\leq a$,
i.e., $x = x\we b$ and $y = y\we a$. Since $(x,a)\in \theta$ and $(y,b)\in \theta$
then $(x\we b,a\we b)\in \theta$ and $(y\we a, b\we a)\in \theta$, i.e.,
$(x, a\we b)\in \theta$ and $(y,a\we b)\in \theta$. Thus, $(x,y)\in \theta$.
Therefore, $(a,b)\in \theta$.
\end{proof}

\begin{cor} \label{corc3}
Let $A\in \gSRL$ and $\theta,\psi \in \ConA$. Then $\theta \subseteq \psi$
if and only if $e/\theta \subseteq e/\psi$.
\end{cor}

\begin{proof}
Let $A\in \gSRL$ and $\theta,\psi \in \ConA$. It is immediate that if
$\theta \subseteq \psi$ then $e/\theta \subseteq e/\psi$. The converse of this
property follows from a straightforward computation based in Lemma \ref{c3}.
\end{proof}

Let $A \in \gSRL$ and $H \in \SU$. We define the following subset of $A\times A$:
\[
\theta_H = \{(a,b) \in A\times A: a\cdot h\leq b\; \text{and}\; b \cdot h\leq a\; \text{for some}\; h\in H\}.
\]

\begin{lem} \label{c4}
Let $A \in \gSRL, H \in \SU$ and $a,b\in A$.
The following conditions are equivalent:
\begin{enumerate}[\normalfont a)]
\item $(a,b)\in \theta_H$.
\item $(a\ra b)\we e\in H$ and $(b\ra a)\wedge e\in H$.
\item $s(a,b)\in H$.
\item There exists $h\in H$ such that $h\leq a\ra b$ and $h\leq b\ra a$.
\end{enumerate}
\end{lem}

\begin{proof}
a) $\Rightarrow$ b). Suppose that $(a,b)\in \theta_H$, so there exists $h\in H$ such that
$a\cdot h\leq b$ and $b \cdot h\leq a$. It follows from Proposition \ref{c1}
that $\square(h)\leq a\ra b$ and $\square(h)\leq b\ra a$. Since $h\in H$
then $\square(h)\in H$ and $\square(h)\we e \in H$. Thus,
$\square(h)\we e \leq  (a\ra b)\we e \leq   e$.
Similarly,
$\square(h)\we e  \leq  (b\ra a)\we e \leq   e$.
By the convexity of $H$ we have that $(a\ra b)\we e \in H$ and $(b\ra a)\we e \in H$.

b) $\Rightarrow$ c). Now assume that $(a\ra b)\we e \in H$ and $(b\ra a)\we e \in H$.
Then $(a\ra b)\we e \we (b\ra a)\we e \in H$, i.e., $s(a,b)\in H$.

c) $\Rightarrow$ d). Suppose that $s(a,b)\in H$. Hence,
$s(a,b) \leq a\ra b$ and $s(a,b) \leq b\ra a$.

d) $\Rightarrow$ a). Finally, suppose that there exists $h\in H$ such that $h\leq a\ra b$
and $h\leq b\ra a$. Thus,
$a \cdot h  \leq   a \cdot (a\ra b) \leq   b$.
Analogously, $b \cdot h \leq   a$.
Hence, $(a,b)\in \theta_H$.
\end{proof}

\begin{lem} \label{c5}
Let $A \in \gSRL$. If $H\in \SU$, then $\theta_H \in \ConA$.
\end{lem}

\begin{proof}
It is routine to prove that $\theta_H$ is an equivalence relation.

Let $(a,b), (c,d)\in \theta_H$. First we will
show that $(a\cdot c,b\cdot d), (a\vee c,b\vee d) \in \theta_H$.
Since $(a,b), (c,d)\in \theta_H$ then it follows from Lemma \ref{c4}
that there exist $h, j\in H$ such that $h\leq a\ra b$, $h\leq b\ra a$,
$j\leq c\ra d$ and $j\leq d\ra c$. Clearly we can replace $h$ and $j$ by $k = h\wedge j$.
Then
\[
(a\cdot c)\cdot k^2 =    (a\cdot k) \cdot (c\cdot k) \leq b\cdot d,
\]
i.e., $(a\cdot c)\cdot k^2 \leq b\cdot d$. In a similar way we can show that
$(b\cdot d)\cdot k^2 \leq a\cdot c$,
so $(a\cdot c, b\cdot d) \in \theta_H$.
Besides
\[
(a\vee c)\cdot k  =    (a\cdot k) \vee (c\cdot k) \leq b \vee d.
\]
Analogously we can show that $(b\vee d)\cdot k \leq a\vee c$, so
$(a\vee c, b\vee d) \in \theta_H$.
Besides,
\[
k \cdot (a\ra c)\cdot k   \leq    (b\ra a) \cdot (a\ra c)\cdot (c\ra d) \leq  b\ra d.
\]
Hence, $(a\ra c)\cdot k^{2} \leq b\ra d$. The inequality
$(b\ra d) \cdot k^{2}\leq a\ra c$ follows in like manner.
Thus, $(a\ra c, b\ra d) \in \theta_H$.

Finally we will show that $(a\we c, b\we d) \in \theta_H$.
In order to prove it, note that
\begin{align*}
(a\wedge c) \ra (b\wedge d)  & =    ((a\wedge c)\ra b)\wedge((a\wedge c)\ra d)\\
                             & \geq  (a\ra b) \wedge (c\ra d)\\
                             & \geq k.
\end{align*}
Analogously, $(b\wedge d)\ra (a\wedge c)\geq k$.
Therefore, $(a\wedge c, b\wedge d)\in \theta_H$.
\end{proof}

\begin{thm}\label{teoc}
If $A$ is a member of $\gSRL$ then $\ConA$ is order isomorphic to $\SUs$.
The order isomorphism is established via the assignments $\theta \mapsto e/\theta$ and
$H\mapsto \theta_H$.
\end{thm}

\begin{proof}
Let $F:\ConA \ra \SUs$ and $G: \SUs \ra \ConA$ be the maps given by $F(\theta) = e/\theta$
and $G(H) = \theta_H$ respectively.
The maps $F$ and $G$ are well-defined by Lemmas \ref{c2} and \ref{c5} respectively.

Let $\theta \in \ConA$. It follows from Lemma \ref{c3} that $\theta = \theta_{e/\theta}$,
i.e., $G(F(\theta)) = \theta$. Let $H\in \SUs$.
We will prove that $H = e/\theta_H$. Let $a\in H$.
Since $H$ is a subalgebra of $A$ and $a\in H$ then $a\ra e, e\ra a, e\in H$.
Hence, $s(a,e)\in H$, so it follows from Lemma \ref{c4} that $(a,e)\in \theta_H$,
i.e., $a\in e/\theta_H$. Thus, $H\subseteq e/\theta_H$. Conversely, let
$a\in e/\theta_H$, i.e., $(a,e) \in \theta_H$. Hence, there exists $h\in H$ such that
$a\cdot h \leq e$ and $h\leq a$. Since $H$ is a strongly convex subalgebra of $A$ we
obtain that $a\in H$. Thus, $e/\theta_H\subseteq H$. We have proved the equality $e/\theta_H = H$,
i.e., $F(G(H)) = H$. Hence, $F$ and $G$ are bijective maps.

It follows from Corollary \ref{corc3} that $F$ is an order isomorphism.
A straightforward computation based in Lemma \ref{c4} shows that if $H_1, H_2\in \SUs$
then $H_1 \subseteq H_2$ if and only if $\theta_{H_1} \subseteq \theta_{H_2}$. Therefore,
$G$ is also an order isomorphism.
\end{proof}

Let $(A,\leq)$ be a poset and $H\subseteq A$. We say that $H$ is an \emph{upset} if for
every $a,b\in A$, if $a\leq b$ and $a\in H$ then $b\in H$.
\vspace{1pt}

The following definition will be used in order to give a simpler description of the
congruences of the algebras of $\igSRL$.

\begin{defn}
Let $A\in \igSRL$ and $H\subseteq A$. We say that $H$ is a filter if $1\in H$, $H$
is an upset and $a\cdot b \in H$ whenever $a,b\in H$.
If in addition $H$ is closed under $\square$ we say that $H$ is a $\square$-filter.
\end{defn}

For $A\in \igSRL$ we write $\mathrm{Fil}(A)$ for the set of filters of $A$ and
$\square \mathrm{Fil}(A)$ for the set of $\square$-filters of $A$.

\begin{lem}\label{lemteoc}
Let $A\in \igSRL$ and $H\subseteq A$. The following conditions are equivalent:
\begin{enumerate}[\normalfont 1)]
\item $H \in \SUs$.
\item $H \in \SU$.
\item $H \in \square\mathrm{Fil}(A)$.
\end{enumerate}
\end{lem}

\begin{proof}
Let $A\in \igSRL$ and $H \subseteq A$.
It is immediate that if $H \in \SUs$ then $H \in \SU$. Conversely,
let $H\in \SU, a\in A$ and $h\in H$
such that $a\cdot h\leq 1$ and $h\leq a$ (in this case the first inequality is redundant).
In particular, $h\leq a \leq 1$. Since $h,1\in H$ then it follows from the convexity of
$H$ that $a\in H$. Hence, $H\in \SUs$. We have proved the equivalence between 1) and 2).

Now we will show the equivalence between 2) and 3).
Let $H \in \SU$. In order to show that $H$ is an upset, let $h\in H$ and $a\in A$ such that
$h\leq a$. Since $a\leq 1$, it follows from the convexity of $H$ that $a\in H$. Thus,
$H \in \square\mathrm{Fil}(A)$. Conversely, suppose that $H \in \square\mathrm{Fil}(A)$. In order to
show that $H$ is a subalgebra of $A$, notice that for $a,b\in H$ we have that $a\cdot b \leq a\we b$
and $a\cdot b \leq a\vee b$, so the fact that $H$ is an upset implies that $a\we b \in H$ and
$a\vee b\in H$. Besides, $\square(b) \leq a\ra b$. Since $\square(b)\in H$ and $H$ is
an upset, $a\ra b\in H$. Hence, $H$ is a subalgebra of $H$. Finally, notice that if $h\leq a\leq k$
for $a\in A$ and $h, k\in H$ then $h\in H$ because $H$ is an upset. Therefore, $H \in \SU$.
\end{proof}

The following result follows from Theorem \ref{teoc} and Lemma \ref{lemteoc}.

\begin{cor}
If $A \in \igSRL$ then $\ConA$ is order isomorphic to $\square\mathrm{Fil}(A)$.
The order isomorphism is established via the assignments $\theta \mapsto 1/\theta$ and
$H\mapsto \theta_H$.
\end{cor}

In what follows we will apply Theorem \ref{teoc} for the particular cases of $\CRL$ and $\SRL$
respectively.

\begin{lem} \label{teoc1}
Let $A\in \CRL$. Then $\SUs = \SU$.
\end{lem}

\begin{proof}
Let $A\in \CRL$. Let $H\in \SU, a\in A$ and $h\in H$
such that $a\cdot h \leq e \leq h\ra a$, i.e., $h\leq a \leq h\ra e$. It follows
from the convexity of $H$ that $a\in H$, so $H\in \SUs$. Thus, $\SU \subseteq \SUs$.
Since $\SUs \subseteq \SU$ then $\SUs = \SU$.
\end{proof}

It follows from Lemma \ref{teoc1} that Theorem \ref{teoc} generalizes \cite[Theorem 2.3]{HRT},
which establishes an order isomorphism between the lattice of congruences
and the lattice of convex subalgebras of each commutative residuated lattice.
Also note that if $A$ is an integral commutative residuated lattice then $\mathrm{Fil}(A) = \square\mathrm{Fil}(A)$,
so it follows from Lemma \ref{lemteoc} that there exists an order isomorphism between the lattice of congruences
and the lattice of filters of each integral commutative residuated lattice.
\vspace{3pt}

Let $A\in \SRL$ and $H$ a subset of $A$. We say that $H$ is a \emph{lattice filter} if
$1\in H$, $H$ is an upset and $a\we b \in H$ whenever $a,b\in H$.
We say that $H$ is an \emph{open lattice filter} if it is a lattice filter
such that $\square(a) \in H$ whenever $a\in H$.
\vspace{1pt}

\begin{prop} \label{teoc2}
Let $A\in \SRL$ (seen as an algebra of $\gSRL$). Then $\SUs$ coincides
with the set of open lattice filters of $A$.
\end{prop}

\begin{proof}
It is a direct consequence of Lemma \ref{lemteoc}.
\end{proof}

It follows from Proposition \ref{teoc2} that Theorem \ref{teoc} also generalizes \cite[Theorem 2]{EH}
(see also  \cite[Theorem 6.12]{CJ}), which establishes
an order isomorphism between the lattice of open lattice filters and the lattice of
congruences of each subresiduated lattice.

\section{The strongly convex subalgebra generated by a subset of the negative cone and some applications} \label{s3}

In this section we give a description of the strongly convex subalgebra generated by a subset of the negative cone
of any srl-monoid. We apply this result in order to give some properties of the lattice of
strongly convex subalgebras of any srl-monoid.
\vspace{1pt}

Let $\mathbb{N}$ be the set of natural numbers.
Let $A\in \gSRL, a\in A$ and $n\in \mathbb{N}$. We define $a^{0} = e$ and $a^{n+1} = a\cdot a^{n}$.
We also define $\square^{0}(a) = a$, $\square^{1}(a) = \square(a)$ and for $n\geq 1$,
$\square^{n+1}(a) = \square{(a)} \cdot \square^{n}(a)$.
\vspace{1pt}

The following lemma follows from a straightforward computation.

\begin{lem} \label{sg0}
Let $A\in \gSRL$ and $a,a_1,\dots,a_n \in A$. Then
\begin{enumerate}[\normalfont 1)]
\item $\square(\square(a_1)\cdot \square(a_2)\cdot \ldots \cdot \square(a_n)) = \square(a_1)\cdot \square(a_2)\cdot \ldots \cdot \square(a_n)$.
\item $\square(a_1)\cdot \square(a_2)\cdot \ldots \cdot \square(a_n) \leq a \ra (a \cdot \square(a_1)\cdot \square(a_2)\cdot \ldots \cdot \square(a_n))$.
\end{enumerate}
\end{lem}

Let $A\in \gSRL$ and $H$ a subalgebra of $A$. We define $H^- = \{a\in H:a\leq e\}$.
This set is called the \emph{negative cone} of $H$.

\begin{rem}\label{remon}
Let $A\in \gSRL, a\in A^-$ and $n_1,n_2,m_1,m_2\in \mathbb{N}$.
A direct computation shows that if $n_1\geq n_2$ and $m_1 \geq m_2$
then $a^{m_1}\leq a^{m_2}$ and $\square^{n_1}(a^{m_1}) \leq \square^{n_2}(a^{m_2})$.
\end{rem}

Let $A\in \gSRL$. For any $S\subseteq A$ we will let $\C[S]$ denote the
smallest strongly convex subalgebra containing $S$ and will let $\C[a] = \C[\{a\}]$.
In what follows, we will let  $ \left\langle S\right\rangle$ denote the
submonoid of $(A, \cdot , e)$ generated by $S$.

The following result is a generalization of \cite[Lemma 2.7]{HRT}.

\begin{lem} \label{lemCS}
Let $A\in \gSRL$ and $S\subseteq A^{-}$. Then
\[
\C[S] = \{x\in A: \square^{n}(h) \leq x\;\text{and}\; x\cdot \square^{n}(h) \leq e,\; \text{for some}\;
h\in \left\langle S\right\rangle \;\text{and}\; n\in \mathbb{N}\}.
\]
\end{lem}

\begin{proof}
Note that if $a\in A^{-}$, then $\square(a) \in A^{-}$.
Let $S\subseteq A^{-}$. Then, $\left\langle S\right\rangle \subseteq A^{-}$.
Put
\[
K = \{x\in A: \square^{n}(h) \leq x\;\text{and}\; x\cdot \square^{n}(h) \leq e,\; \text{for some}\;
h\in \left\langle S\right\rangle \;\text{and}\; n\in \mathbb{N}\}.
\]
It is clear that $S\subseteq K \subseteq \C[S]$.
It will suffice to show that $K$ is a strongly convex subalgebra of $A$.
Let $a, b\in K$. Thus, there are $h_{a}, h_{b} \in \left\langle S\right\rangle$ and
$n, m\in \mathbb{N}$ such that
$\square^{n}(h_{a}) \leq a, a \cdot\square^{n}(h_{a}) \leq e,
\square^{m}(h_{b})\leq b$ and $b\cdot \square^{m}(h_{b}) \leq e$.
We may replace $h_{a}$ and $h_{b}$ by $h = h_{a}\cdot h_{b}$, and $n,m$ by $k =$ max $\{n,m,1\}$.

On the other hand, $\square^{k}(h)  \leq   a\we b \leq   a\vee b$ and
\begin{align*}
\square^{k}(h) \cdot (a\we b) & \leq   \square^{k}(h) \cdot (a\vee b)\\
                              & =      (\square^{k}(h)\cdot a) \vee (\square^{k}(h)\cdot b)\\
                              & \leq   e,
\end{align*}
so $K$ is closed under meets and joins. Besides,
\[
\square^{2k}(h)  =   \square^{k}(h)\cdot \square^{k}(h) \leq  a\cdot b
\]
and
\[
\square^{2k}(h)\cdot (a\cdot b)  =   (\square^{k}(h)\cdot a)\cdot (\square^{k}(h) \cdot b) \leq e,
\]
so $K$ is also closed under product.

Now we will show that $K$ is closed under the arrow.
First note that
\[
a\cdot \square^{2k}(h)  =     (a\cdot \square^{k}(h))\cdot \square^{k}(h) \leq  \square^{k}(h) \leq  b,
\]
so
\begin{equation} \label{sg1}
a \ra (a\cdot \square^{2k}(h)) \leq a\ra b.
\end{equation}
Besides, it follows from Lemma \ref{sg0} that
\begin{equation} \label{sg2}
\square^{2k}(h) \leq a \ra (a\cdot \square^{2k}(h)).
\end{equation}
Thus, it follows from (\ref{sg1}) and (\ref{sg2}) that
\begin{equation} \label{sg3}
\square^{2k}(h) \leq a\ra b.
\end{equation}
We also have that
\begin{align*}
\square^{2k}(h)\cdot (a\ra b) & =     \square^{k}(h)\cdot \square^{k}(h) \cdot (a\ra b)\\
                              & \leq  \square^{k}(h)\cdot a\cdot (a\ra b)\\
                              & \leq  \square^{k}(h)\cdot b\\
                              & \leq  e,
\end{align*}
so
\begin{equation} \label{sg4}
\square^{2k}(h)\cdot (a\ra b) \leq e.
\end{equation}
Hence, it follows from (\ref{sg3}) and (\ref{sg4}) that $a\ra b\in K$.
Thus, $K$ is a subalgebra of $A$.

Finally we will see that $K$ is a strongly convex subalgebra of $A$.
The set $K$ is convex because if $a\leq x\leq b$ then $\square^{k}(h)\leq a \leq x$ and
$a\cdot \square^{k}(h) \leq x \cdot \square^{k}(h)\leq b \cdot \square^{k}(h) \leq e$,
so $\square^{k}(h)\leq x$ and $x \cdot \square^{k}(h) \leq e$, i.e., $x\in K$. We have proved the convexity of $K$.
Consider $x\in A$ such that $x\cdot a \leq e$ and $a\leq x$ for some $a\in K$. Since $\square^{k}(h)\leq a\leq x$
then $\square^{k}(h)\leq x$. Besides, $x\cdot \square^{k}(h) \leq a\cdot x \leq e$,
so $x\cdot \square^{k}(h) \leq e$. Then $x\in K$. Therefore, $K$ is a strongly convex subalgebra of $A$.
\end{proof}

\begin{cor} \label{coroc}
Let $A\in \gSRL$ and $a\in A^{-}$. Then
\[
\C[a] = \{x\in A: \square^{n}(a^{m}) \leq x\;\text{and}\; x\cdot \square^{n}(a^{m}) \leq e,\;
\text{for some}\: n,m\in \mathbb{N}\}.
\]
Moreover, if $x\in A^{-}$ then $x\in \C[a]$ if and only if there are $n,m$
such that $\square^{n}(a^{m}) \leq x$.
\end{cor}

\begin{rem} \label{remon1}
It follows from Remark \ref{remon} that
the conclusions of Corollary \ref{coroc} also hold by considering $n=m$.
\end{rem}

Let $A\in \gSRL$ and $S\subseteq A\times A$. We write $\theta(S)$ for the smallest
congruence $\theta$ such that $S\subseteq \theta$. If $S=\{(a_1,b_1),\dots, (a_n,b_n)\}$
we write $\theta((a_1,b_1),\dots,(a_n,b_n))$ in place of $\theta(\{(a_1,b_1),\dots, (a_n,b_n)\})$.
If $S = \{(a,b)\}$ we write $\theta(a,b)$ in place of $\theta(\{(a,b)\})$.

\begin{lem}\label{pc1}
Let $A\in \gSRL$ and $a,b\in A$. Then $e/\theta(a,b) = \C[s(a,b)]$.
\end{lem}

\begin{proof}
Let $A\in \gSRL$ and $a,b\in A$. Taking into account Theorem \ref{teoc} we have that
\[
e/\theta(a,b)  =  \bigcap _{(a,b) \in \theta} e/\theta
               =  \bigcap_{s(a,b) \in e/\theta} e/\theta
               =  \C[s(a,b)].
\]
\end{proof}

Let $A\in \gSRL$, $S\subseteq A^{-}$ and $T\subseteq A^{-}$. Notice that it follows from Lemma \ref{lemCS}
that if $S\subseteq T$ then $\C[S] \subseteq \C[T]$.

\begin{thm} \label{distlat}
Let $A \in \gSRL$. Then $\SUs$ is an algebraic and distributive lattice such that for every
$a,b\in A^{-}$ the equality $\C[a\we b] = \C[a] \vee \C[b]$ is satisfied.
Moreover, the compact elements of $\SUs$ are the elements of the form $\C[a]$ for $a\in A^{-}$.
\end{thm}

\begin{proof}
Let $A\in \gSRL$. It follows from Theorem \ref{teoc} that $\SUs$ is an algebraic and distributive lattice.

Let $a,b\in A^{-}$. Notice that $\C[a] \vee \C[b] = \C[\{a,b\}]$. In what follows
we will see that $\C[a\we b] = \C[\{a,b\}]$. Since $a\we b \in  \C[\{a,b\}]$
then $\C[a\we b] \subseteq \C[\{a,b\}]$. Conversely, since $a\we b\leq a\leq e$ and
$a\we b\leq b\leq e$ then $a\in \C[a\we b]$ and $b\in \C[a\we b]$, i.e., $\{a,b\} \subseteq \C[a\we b]$.
Hence, $\C[\{a,b\}] \subseteq \C[a\we b]$. Therefore,
\begin{equation} \label{eqsup}
\C[a\we b] = \C[a] \vee \C[b].
\end{equation}
It is known that the compact elements of $\ConA$ are the finitely generated members
$\theta((a_1,b_1),\dots,(a_n,b_n))$ of $\ConA$, so the compact elements of $\SUs$ are that
of the form $e/\theta((a_1,b_1),\dots,(a_n,b_n))$. But
\[
\theta((a_1,b_1),\dots,(a_n,b_n)) = \theta(a_1,b_1)\vee \cdots \vee \theta(a_n,b_n),
\]
so it follows from  Theorem \ref{teoc},
Lemma \ref{pc1} and (\ref{eqsup}) that
\[
e/\theta((a_1,b_1),\dots,(a_n,b_n)) = \C[s(a_1,b_1)\we \dots \we s(a_n,b_n)].
\]
Therefore, the
compact elements of $\SUs$ are that of the form $\C[a]$ for $a\in A^{-}$.
\end{proof}

It is natural to ask the following question: for every
$A\in \gSRL$ and $a,b \in A^-$, does it hold the equality $\C[a\vee b] = \C[a] \cap \C[b]$?
The answer is negative. Indeed, let $A$ be the bounded lattice of four elements, where $0$ is the bottom, $1$ is the top
and $a,b$ are incomparable elements. Let $Q = \{0,1\}$.
We have that $(A,Q)$ is a subresiduated lattice, where the binary operation $\ra$ is determined by

\[%
\begin{tabular}
[c]{c|cccc}%
$\ra$ & $0$        &  $a$    &  $b$       & $1$ \\\hline
$0$     & $1$        &  $1$    &  $1$       & $1$ \\
$a$     & $0$        &  $1$    &  $0$       & $1$ \\
$b$     & $0$        &  $0$    &  $1$       & $1$ \\
$1$     & $0$        &  $0$    &  $0$       & $1$ %
\end{tabular}
\hskip15pt
\]

Thus, this subresiduated lattice can be seen as an algebra of $\gSRL$. Moreover, for $c\in \{0,a,b,1\}$
we have that $\C[c]$ is the open lattice filter generated by $c$, i.e., the least open lattice
filters (with respect to the inclusion) which contains to the element $c$.
It is immediate that $\C[1] = \{1\}$ and $\C[a] = \C[b] = \{0,a,b,1\}$,
so $\C[a\vee b] \neq \C[a] \cap \C[b]$.
\vspace{3pt}

We finalize this section with a characterization of the principal congruences of the
algebras of $\gSRL$.

\begin{thm} \label{thmpc}
Let $A\in \gSRL$ and $a,b \in A$. Then $(x,y)\in \theta(a,b)$ if and only if there exist natural numbers
$n$ and $m$ such that $\square^{m}(s(a,b)^n)\leq s(x,y)$.
\end{thm}

\begin{proof}
It follows from Lemma \ref{c3}, Corollary \ref{coroc} and Lemma \ref{pc1}.
\end{proof}

\begin{rem}
It follows from Theorem \ref{thmpc} and Remark \ref{remon1}
that if $A\in \gSRL$ and $a,b \in A$,
$(x,y)\in \theta(a,b)$ if and only if there exists $n\in \mathbb{N}$
such that $\square^{n}(s(a,b)^n)\leq s(x,y)$.
\end{rem}

\section{The variety generated by the totally ordered algebras of $\gSRL$} \label{s4}

We write $\gSRLc$ for the subvariety of $\gSRL$ generated by its totally
ordered members.
The aim of this section is to give alternative equational basis for the
variety $\gSRLc$.
\vspace{1pt}

Consider the following identities:
\begin{enumerate}
\item[$\Co$] $e\leq (x\ra y)\vee (y\ra x)$,
\item[$\Ct$] $e\we (x\vee y) = (e\we x) \vee (e\we y)$,
\item[$\Eo$] $(x\we y)\ra z = (x\ra z) \vee (y\ra z)$,
\item[$\Et$] $z\ra (x\vee y) = (z\ra x) \vee (z\ra y)$.
\end{enumerate}

\begin{lem} \label{laf1}
Let $A\in \gSRL$.
\begin{enumerate}[\normalfont a)]
\item If $A$ satisfies $\Eo$ or $\Et$ then $A$ satisfies $\Co$.
\item If $A$ satisfies $\Co$ then  for every $a,b \in A^{-}$ and $m\in \mathbb{N}$,
$(a\vee b)^{2^m} = a^{2^m} \vee b^{2^m}$.
\item If $A$ satisfies $\Et$ then for every $a,b \in A^{-}$ and $n,m\in \mathbb{N}$,
$\square^{2^n}((a\vee b)^{2^m}) = \square^{2^n}(a^{2^m}) \vee \square^{2^n}(b^{2^m})$.
\end{enumerate}
\end{lem}

\begin{proof}
a) First suppose that $A$ satisfies $\Eo$ and let $a,b\in A$.
Then
\begin{align*}
e     & \leq  (a\we b)\ra (a\we b)  \\
      & =      (a\ra (a\we b)) \vee (b\ra (a\we b))\\
      & \leq    (a\ra b) \vee (b\ra a),
\end{align*}
so $e\leq (a\ra b)\vee (b\ra a)$.
Similarly it can be proved that if $A$ satisfies $\Et$ then $A$ satisfies $\Co$.

b) Assume that $A$ satisfies $\Co$ and let $a,b\in A^-$. Then
\[
(a\vee b)^2 = a^2 \vee b^2 \vee (a\cdot b).
\]
Since $e\leq (a\ra b) \vee (b\ra a)$ then
\begin{align*}
a\cdot b     & \leq  [b\cdot a \cdot (a\ra b)] \vee [a\cdot b \cdot (a\ra b)]  \\
             & \leq   b^2\vee a^2,
\end{align*}
so
\begin{equation} \label{eqt}
(a\vee b)^2 = a^2 \vee b^2.
\end{equation}
Now suppose that $(a\vee b)^{2^m} = a^{2^m} \vee b^{2^m}$ for some $m$.
Since $a^{2^m}, b^{2^m} \in A^-$ then it follows from (\ref{eqt}) that
\begin{align*}
(a\vee b)^{2^{m+1}}     & = [(a\vee b)^{2^m}]^2  \\
                        & = [a^{2^m} \vee b^{2^m}]^2   \\
                        & = a^{2^{m+1}} \vee b^{2^{m+1}}.
\end{align*}
c) Finally assume that $A$ satisfies $\Et$ and let $a,b\in A^-$ and $n,m\in \mathbb{N}$.
First we will see that $\square((a\vee b)^{2^m}) = \square(a^{2^m}) \vee \square(b^{2^m})$.
Indeed, it follows from b) and $\Et$ that
\begin{align*}
\square((a\vee b)^{2^m})     & = \square(a^{2^m} \vee b^{2^m})  \\
                             & = \square(a^{2^m}) \vee \square(b^{2^m}).
\end{align*}
Thus,
\begin{equation} \label{eqs}
\square((a\vee b)^{2^m}) = \square(a^{2^m}) \vee \square(b^{2^m}).
\end{equation}
Hence, since  $\square(a^{2^m}),  \square(b^{2^m}) \in A^-$, by b) and (\ref{eqs})
we get
\begin{align*}
\square^{2^n}((a\vee b)^{2^m})     & = [\square(a^{2^m}) \vee \square(b^{2^m})]^{2^n}  \\
                               & = \square^{2^n}(a^{2^m}) \vee \square^{2^n}(b^{2^m}),
\end{align*}
which was our aim.
\end{proof}

\begin{lem} \label{leminf}
Assume that $A\in \gSRL$ satisfies $\Et$.
If $a,b\in A^-$ then $\C[a\vee b] = \C[a] \cap \C[b]$.
\end{lem}

\begin{proof}
Consider $a,b\in A^{-}$.
Since $a\leq a\vee b\leq e$ and $b\leq a\vee b\leq e$ then $a\vee b \in \C[a]\cap \C[b]$,
so $\C[a\vee b] \subseteq \C[a] \cap \C[b]$. Conversely, let $x\in \C[a] \cap \C[b]$. Thus,
it follows from Corollary \ref{coroc} that there exist $n_1,n_2,m_1,m_2\in \mathbb{N}$
such that $\square^{n_1}(a^{m_1}) \leq x$,
$\square^{n_2}(b^{m_2}) \leq x$, $x\cdot \square^{n_1}(a^{m_1}) \leq e$ and
$x\cdot \square^{n_2}(b^{m_2}) \leq e$. Take $n =\; max\;\{n_1,n_2,1\}$ and
$m =\; max\;\{m_1,m_2,1\}$. Hence, $\square^{n}(a^{m}) \leq x$,
$x\cdot \square^{n}(a^{m}) \leq e$, $\square^{n}(b^{m}) \leq x$ and
$x\cdot \square^{n}(b^{m}) \leq e$. Moreover, since $n\leq 2^n$ and $m\leq 2^m$ then $\square^{2^n}(a^{2^m}) \leq x$,
$x\cdot \square^{2^n}(a^{2^m}) \leq e$, $\square^{2^n}(b^{2^m}) \leq x$ and
$x\cdot \square^{2^n}(b^{2^m}) \leq e$.
Then, it follows from item c) of Lemma \ref{laf1} that
\[
\square^{2^n}((a \vee b)^{2^m}) = \square^{2^n}(a^{2^m}) \vee \square^{2^n}(b^{2^m}),
\]
so
$\square^{2^n}((a \vee b)^{2^m}) \leq x$ because $\square^{2^n}(a^{2^m}) \leq x$ and $\square^{2^n}(b^{2^m}) \leq x$.

Besides,
\begin{align*}
x\cdot \square^{2^n}((a\vee b)^{2^m})   & = x\cdot (\square^{2^n}(a^{2^m}) \vee \square^{2^n}(b^{2^m})) \\
                                        & = (x\cdot \square^{2^n}(a^{2^m})) \vee (x\cdot \square^{2^n}(b^{2^m}))\\
                                        & \leq e,
\end{align*}
so $x\cdot \square^{2^n}((a\vee b)^{2^m})\leq e$. Thus, again by Corollary \ref{coroc} we get
$x\in \C[a\vee b]$. Therefore, $\C[a\vee b] = \C[a] \cap \C[b]$.
\end{proof}

The following result follows from Theorem \ref{distlat} and Lemma \ref{leminf}.

\begin{cor} \label{proppre}
Assume that $A\in \gSRL$ satisfies $\Et$.
Then $\SUs$ is an algebraic and distributive lattice whose compact elements are that of the form $\C[a]$
for $a\in A^{-}$. Moreover, the compact elements of $\SUs$ form a sublattice,
with joins and meets given by the following formulas for every $a,b\in A^{-}$:
\[
\C[a\we b] = \C[a] \vee \C[b]\;\;\;\text{and}\;\;\;\C[a\vee b] = \C[a] \cap \C[b].
\]
\end{cor}

\begin{rem}
The totally ordered members of $\gSRL$ satisfy $\Et$.
Hence, the members of $\gSRLc$ satisfy $\Et$.
\end{rem}

The structure of the proof of the following theorem is like that done for \cite[Theorem 3.1]{HRT}.
However, by completeness we opt to give a self contained proof of it.

\begin{thm} \label{rit}
The identities $\Ct$ and $\Et$,
together with those defining $\gSRL$, form an equational basis for $\gSRLc$.
\end{thm}

\begin{proof}
Let $\mathcal{V}$ be denote the subvariety of $\gSRL$ determined by identities $\Ct$ and $\Et$.
Since any totally ordered member of $\gSRL$ clearly satisfies these identities, it
follows that $\gSRLc \subseteq \mathcal{V}$.

To prove the converse it will suffice
to prove that every subdirectly irreducible member of $\mathcal{V}$ is totally ordered. We prove
the contrapositive. Suppose that $A$ satisfies identities $\Ct$ and $\Et$
but that it is not totally ordered. Let $a$ and $b$ incomparable elements in $L$, i.e.,
$e\nleq a\ra b$ and $e\nleq b\ra a$. Let $u= e\we (a\ra b)$ and $v = e\we (b\ra a)$. Note that $u\neq e$ and
$v\neq e$. By identity $\Ct$, $u\vee v = e\we [(a\ra b) \vee (b\ra a)]$.
Since $A$ satisfies $\Et$ then it follows from a) of Lemma \ref{laf1} that $A$ satisfies $\Co$.
Thus, by identity $\Co$, $u\vee v = e$.
Hence, it follows from c) of Lemma \ref{laf1} and Lemma \ref{leminf} that
\begin{align*}
\C[u] \cap \C[v]    & =  \C[u\vee v]  \\
                    & =  \C[e]\\
                    & = \{e\}.
\end{align*}
Then $\C[u] \cap \C[v] = \{e\}$ and in consequence $\ConA$ cannot have a monolith. Then, $A$ is
not subdirectly irreducible. Therefore, $\mathcal{V} \subseteq \gSRLc$.
\end{proof}

The following result is similar to \cite[Theorem 3.4]{HRT}.

\begin{cor} \label{corgc}
Coupled with the identities defining $\gSRL$, both $\{\Eo,\Ct\}$ and
$\{\Co,\Ct\}$ form alternative equational basis for $\gSRLc$.
\end{cor}

\begin{proof}
Let $\mathcal{V}$ be the subvariety of $\gSRL$ which satisfies the equations $\Eo,\Ct$
and $\mathcal{W}$ the subvariety of $\gSRL$ which satisfy the equations $\Co,\Ct$.
Since every chain in $\gSRL$ satisfies $\Eo,\Ct$, it follows from Theorem \ref{rit}
that $\gSRLc \subseteq \mathcal{V}$. The inclusion $\mathcal{V} \subseteq \mathcal{W}$
follows from Lemma \ref{laf1}.

Finally we will see that $\mathcal{W} \subseteq \gSRLc$. Let $A\in \mathcal{W}$
and $a,b,c \in A$. Define $u = (a\ra b) \vee (a\ra c)$. It is immediate that
$u\leq a\ra (b\vee c)$. Our aim is to show the other inequality, which is equivalent to
showing that
\[
e \leq [ a\ra (b\vee c)]\ra u.
\]
In order to see it, first notice that
$(b\vee c) \ra b \leq [(a\ra (b\vee c)] \ra u$ if and only if
\[
[(a\ra (b\vee c)] \cdot [(b\vee c) \ra b] \leq u.
\]
But
\begin{align*}
[(a\ra (b\vee c)] \cdot [(b\vee c) \ra b]    & \leq  a\ra b  \\
                                             & \leq (a\ra b) \vee (a\ra c)\\
                                             & = u,
\end{align*}
so $[(a\ra (b\vee c)] \cdot [(b\vee c) \ra b] \leq u$. Hence,
\[
(b\vee c) \ra b \leq [(a\ra (b\vee c)] \ra u.
\]
Similarly, we have that
$[(a\ra (b\vee c)] \cdot [(b\vee c) \ra c] \leq u$,
so
\[
(b\vee c) \ra c \leq [(a\ra (b\vee c)] \ra u.
\]

Thus, taking into account $\Co$ and $\Ct$ we obtain that

\begin{align*}
[(a\ra (b\vee c)] \ra u      & \geq  [(b\vee c) \ra b] \vee [(b\vee c) \ra c] \\
                             & = [(b\ra b) \we (c\ra b)] \vee [(c\ra c) \we (b\ra c)]\\
                             & \geq [(e \we (c\ra b)] \vee [e \we (b\ra c)]\\
                             & = e\we  [(c\ra b) \vee (b\ra c)]\\
                             & = e,
\end{align*}
so $e\leq [(a\ra (b\vee c)] \ra u$, i.e., $a\ra (b\vee c)\leq u$.
\end{proof}

Let $A$ be a totally ordered member of $\gSRL$. A straightforward computation
shows that for every $a,b,c\in A$, $(a\vee b) \we c = (a\we c) \vee (b\we c)$
and $a\cdot (b\we c) = (a\cdot b) \we (a\cdot c)$. Thus, these two equations
are also satisfied in $\gSRLc$. The fact that the underlying lattice of an algebra
of $\gSRLc$ is distributive can be also proved by using Corollary \ref{corgc} and making the
proof done in \cite[Theorem 3.5]{HRT}.
\vspace{1pt}

Finally, note that Corollary \ref{corgc} in the framework of commutative residuated lattices
is \cite[Theorem 3.4]{HRT}. We also have that Corollary \ref{corgc} implies that the subvariety
of $\SRL$ generated by its totally ordered members is characterized by the equation
$(a\ra b)\vee (b\ra a) = 1$ or, equivalently, by the equations $\Eo$ or
$\Et$ (these properties can be also obtained from results of \cite{Cb}).

\section{Conclusions}

In this paper we have introduced and studied subresiduated lattice ordered commutative monoids
(or srl-monoids for short) and integral srl-monoids.
The definition of srl-monoid was motivated by the definitions of
subresiduated lattice (or sr-lattice for short) \cite{EH}
and commutative residuated lattice \cite{HRT} respectively.
\vspace{1pt}

The following elemental properties establish the connection between sr-lattices, commutative
residuated lattices and srl-monoids:
\begin{itemize}
\item If $(A,\we,\vee,\ra,0,1)$ is a sr-lattice then $(A,\we,\vee,\we,\ra,1)$ is a srl-monoid
(in this sense we say that every sr-lattice is an srl-monoid). Moreover,
$(A,\we,\vee, \ra,0,1)$ is a sr-lattice if and only if $(A,\we,\vee,\we,\ra,1)$ is an integral srl-monoid
and this algebra has a least element.
\item If $(A,\we,\vee,\cdot,\ra,e)$ is a commutative residuated lattice then the algebra $(A,\we,\vee,\cdot,\ra,e)$ is
a srl-monoid (it can be proved by defining $Q = A$). Moreover,
$(A,\we,\vee,\cdot,\ra,e)$ is a commutative residuated lattice if and only if
$(A,\we,\vee,\cdot,\ra,e)$ is a srl-monoid and $e\ra a = a$ for every $a\in A$.
\end{itemize}

The aim of the present paper was to generalize results on
sr-lattices and commutative residuated lattices respectively
in the framework of srl-monoids. In particular, we have showed that the class of
srl-monoids is a variety. We have also studied the lattice of congruences
of any srl-monoid and as an application it was proved that
coupled with the identities defining srl-monoids, both $\{\Eo,\Ct\}$ and
$\{\Co,\Ct\}$ form alternative equational basis for the the variety of srl-monoids generated
by its totally ordered members, where the identities $\Co, \Ct, \Eo$ and $\Et$
are defined by
\begin{enumerate}
\item[$\Co$] $e\leq (x\ra y)\vee (y\ra x)$,
\item[$\Ct$] $e\we (x\vee y) = (e\we x) \vee (e\we y)$,
\item[$\Eo$] $(x\we y)\ra z = (x\ra z) \vee (y\ra z)$,
\item[$\Et$] $z\ra (x\vee y) = (z\ra x) \vee (z\ra y)$.
\end{enumerate}
\vspace{1pt}

In a future paper we will study the class whose members are the
$\{\we,\cdot,\ra,1\}$-subreducts of integral srl-monoids. We will
also study the features of the logics semantically defined by the class previously
mentioned and the variety of integral srl-monoids.

\subsection*{Acknowledgments}

The authors thank Jos\'e Luis Castiglioni for several
conversations concerning the matter of this paper.
This work was supported by Consejo Nacional de Investigaciones Cient\'ificas y T\'ecnicas
(PIP 11220170100195CO and PIP 11220200100912CO, CONICET-Argentina), Universidad Nacional del Sur (PGI24/LZ18),
Universidad Nacional de La Plata (11X/921) and Agencia Nacional de Promoción Científica y
Tecnológica (PICT2019-2019-00882, ANPCyT-Argentina).

-----------------------------------------------------------------------------------------
\\
Juan Manuel Cornejo,\\
Departamento de Matem\'atica, \\
Universidad Nacional del Sur, \\
and CONICET.\\
Av. Leandro N. Alem Nº1253 - 2º Piso,\\
Bah\'ia Blanca (8000),\\
Argentina.\\
jmcornejo@uns.edu.ar

-----------------------------------------------------------------------------------------
\\
Corresponding author:\\
Hern\'an Javier San Mart\'in,\\
Departamento de Matem\'atica, \\
Facultad de Ciencias Exactas (UNLP), \\
and CONICET.\\
Casilla de correos 172,\\
La Plata (1900), Argentina.\\
hsanmartin@mate.unlp.edu.ar

-----------------------------------------------------------------------------------------
\\
Valeria Anah\'i S\'igal,\\
Departamento de Matem\'atica, \\
Facultad de Ciencias Exactas (UNLP), \\
and CONICET.\\
Casilla de correos 172,\\
La Plata (1900), Argentina.\\
vsigal@mate.unlp.edu.ar


{\small \begin{thebibliography}{99}

\bibitem{CSM}
Castiglioni J.L. and San Mart\'{\i}n H.J.,
\emph{l-Hemi-Implicative-Semilattices}.
Studia Logica 106, 675--690 (2018).

\bibitem{C}
Celani S.A., \emph{Distributive Lattices with Fusion and
Implication}. Southeast Asian Bulletin of Mathematics 28, 999-1010
(2004).

\bibitem{Cb}
Celani S.A., \emph{n-linear weakly Heyting algebras}. Mathematical
Logic Quarterly 52, no. 4, 404--416 (2006).

\bibitem{CJ}
Celani S.A. and Jansana R., \emph{Bounded distributive lattices
with strict implication}. Mathematical Logic Quarterly 51, no. 3,
219--246 (2005).

\bibitem{CNSM}
Celani S.A., Nagy A.L. and San Mart\'{\i}n H.J.,
\emph{Dualities for subresiduated lattices}.
Algebra Universalis, vol. 82, no. 59 (2021).
https://doi.org/10.1007/s00012-021-00752-3.

\bibitem{COM}
Cignoli R., D'Ottaviano I. and Mundici D., \emph{Algebraic
Foundations of Many-valued Reasoning}. Trends in Logic-Studia
Logica library (2000).

\bibitem{EH}
Epstein G. and Horn A., \emph{Logics which are characterized by
subresiduated lattices}. Mathematical Logic Quarterly, vol. 22,
no. 1, 199--210 (1976).

\bibitem{EG}
Esteva F. and Godo L., \emph{Monoidal t-norm based logic: towards a logic for
left-continuous t-norms}. Fuzzy Sets and Systems, vol. 124, no. 3, 271--288 (2001).

\bibitem{Font}
Font J.M., \textit{Abstract Algebraic Logic. An introductory Textbook}.
Studies in Logic 60, College Publ. (2016).

\bibitem{GJKO}
Galatos N., Jipsen P., Kowalski T. and Ono H.,
\emph{Residuated Lattices: An Algebraic Glimpse at Substructural Logics}.
Elsevier B. V., Amsterdam ISBN:9780444521415 (2007).

\bibitem{HRT}
Hart J., Raftery L. and Tsinakis C.,
\emph{The structure of commutative residuated lattices}.
Internat. J. Algebra Comput. 12, 509--524 (2002).

\bibitem{J}
Jasem  M., \emph{On ideals of lattice ordered monoids}.
Mathematica Bohemica, vol. 132, no. 4, 369--387 (2007).

\bibitem{SM0}
San Mart\'{\i}n H.J., \emph{Compatible operations in some subvarieties of the variety
of weak Heyting algebras}. In: Proceedings of the 8th Conference of the Eur.
Soc. for Fuzzy Logic and Technology (EUSFLAT 2013). Advances in
Inteligent Systems Research, pp. 475--480. Atlantis Press (2013).

\bibitem{W}
Wille A.M., \emph{The Variety of Lattice-Ordered Monoids Generated by the Natural
Numbers}. Studia Logica 76, 275--290 (2004).


\end{thebibliography}}
\end{document}